\documentclass[11pt]{article}
\usepackage[letterpaper]{geometry}
\usepackage{amsfonts}
\usepackage{amssymb}
\usepackage{amsmath}
\usepackage{latexsym}
\usepackage{amsthm}
\usepackage{epsfig}
\usepackage{graphicx}
\usepackage{subfig}
\usepackage{color}

\usepackage{pstricks}
\usepackage{multido}
\usepackage{pst-node}

\usepackage{booktabs}
\usepackage{multirow}
\usepackage{array}
\usepackage{colortbl}
\usepackage{xcolor}

\newtheorem{thm}{Theorem}[section]

\newtheorem{prop}[thm]{Proposition}

\newcommand{\old}[1]{{}}

\begin{document}

\title{Experiments with two-row cuts from degenerate tableaux}

\author{
Amitabh Basu${}^{1,2}$, \;
Pierre Bonami${}^{3,4}$\\
G\'erard Cornu\'ejols${}^{1,3,5}$,\;
Fran\c{c}ois Margot${}^{1,6}$
}
\date{November 2009}

\maketitle

\begin{abstract}
There has been a recent interest in cutting planes generated from
two or more rows of the optimal simplex tableau. One can construct
examples of integer programs for which a single cutting plane
generated from two rows dominates the entire split closure.
Motivated by these theoretical results, we study the effect of
adding a family of cutting planes generated from two rows on a set
of instances from the MIPLIB library. The conclusion of whether
these cuts are competitive with GMI cuts is very sensitive to the
experimental setup. In particular, we consider the issue of
reliability versus aggressiveness of the cut generators, an issue
that is usually not addressed in the literature.
\end{abstract}

\footnotetext[1] {Tepper School of Business, Carnegie Mellon University,
Pittsburgh, PA 15213.}

\footnotetext[2] {Supported by a Mellon Fellowship.}

\footnotetext[3] {LIF, Facult\'e des Sciences de Luminy,
Universit\'e de Marseille, France.}

\footnotetext[4] {Supported by ANR grant ANR06-BLAN-0375.}

\footnotetext[5] {Supported by  NSF grant CMMI0653419, ONR grant
N00014-03-1-0133 and ANR grant ANR06-BLAN-0375.}

\footnotetext[6] {Supported by  ONR grant N00014-03-1-0133.}

\section{Introduction}
In the last 15 years, generic cutting planes have played a major
role in the progress of mixed integer linear programming (MILP)
solvers. Most cutting plane algorithms available in state of the art
solvers rely on cuts that can be derived from a single equation
(note that the equation used to derive cuts need not be one of the
constraints of the problem; it may be obtained by a linear
combination of the constraints). Examples are Gomory Mixed Integer
cuts~\cite{go}, MIR cuts~\cite{mw}, lift-and-project
cuts~\cite{bcc}, or lifted cover inequalities~\cite{cjp}. This
general family of cuts is known as the family of {\it split cuts}
\cite{COOK}. A natural extension is to derive cuts from more than
one equation, an area of research enjoying a revival in recent
years.

The study of the Corner Polyhedron started by Gomory and
Johnson~\cite{gomory, gj, johnson}, and of Intersection Cuts
introduced by Balas \cite{bal} provides a framework for
generating cutting planes using multiple equations. Several papers
give theoretical results on cuts that can be derived from two
equations~\cite{alww,BBCM,bc,cm,dw}. In \cite{BBCM}, we
 provide examples of integer programs with two
constraints and two integer variables, where the integrality gap can
be closed by a single inequality derived from the two constraints,
while the value of the linear relaxation obtained by adding all
split cuts is arbitrarily close to 0. These examples suggest that certain
cuts derived from two equations  might
improve the efficiency of MILP solvers. In this paper, we test
the empirical effectiveness of this class of two-row cuts.

The paper is organized as follows. In Section~\ref{sec:2row}, we
recall basic results on two-row cuts and motivate the choice of the
triangles used in our separation algorithm. In
Section~\ref{sec:derivation}, we give closed form formulas for
certain types of cuts, discuss integer lifting of these cuts, as
well as their strengthening when the basic variables are
nonnegative. In Section~\ref{sec:compu}, we present computational
results. We show that the conclusions are very sensitive to the
experimental setup. We present two different setups which lead to
opposite conclusions and argue that the setup based on {\tt
CglGomory} is flawed because it overlooks the issues of reliability
versus aggressiveness of the cut generator.

\section{Two-Row Cutting Planes}
\label{sec:2row}
\subsection{Basic Theory}
\label{sec:theory}
In this paper, we study a class of cutting planes derived from two equations.
Consider a mixed-integer problem with two free integer variables
and a finite number of nonnegative, continuous variables.
\begin{equation} \label{SI}
\begin{array}{rrcl}
          & x & = & f + \displaystyle \sum_{j=1}^{k} r^j s_j        \\[0.1cm]
           &  x  & \in & \mathbb{Z}^2                  \\[0.1cm]
          &  s  & \in & \mathbb{R}_+^k \ .
\end{array}
\end{equation}

We assume $f \in \mathbb{Q}^2 \setminus \mathbb{Z}^2$, $k \ge 1$,
and for $j = 1, \ldots, k$,
$r^j \in \mathbb{Q}^2 \setminus \left\{ 0 \right\}$.

Model (\ref{SI}) naturally arises as a relaxation of a MILP having a
basic feasible solution of its linear relaxation with at least two
basic integer variables, at least one of which takes a fractional value in
the optimal solution. Indeed, consider the simplex tableau
corresponding to the basic solution and keep only two equations
corresponding to these two basic integer variables. Model (\ref{SI})
was introduced by Andersen, Louveaux, Weismantel, and
Wolsey~\cite{alww}. Several recent papers discuss some of its
theoretical properties such as characterization of facets and
relative strength of classes of facets ~\cite{bc,BBCM,cm,dw} but
little empirical evidence on the strength of the resulting cuts is
available~\cite{esp}. The remainder of this section is a short
summary of results obtained in \cite{bc,BBCM,cm,dw} that are
directly relevant to the work in this paper.

The inequalities $s_j \ge 0$ are called {\em trivial} for (\ref{SI}). A nontrivial valid inequality for (\ref{SI}) is of the form
\begin{equation} \label{valid}
\sum_{j=1}^k \psi(r^j) s_j \geq 1 \ ,
\end{equation}
where $\psi: \mathbb{R}^2 \rightarrow \mathbb{R}_+$. A
nontrivial valid inequality is {\em minimal} if there is no other
nontrivial valid inequality $\sum_{j=1}^k \psi'(r^j) s_j \geq 1$
such that $\psi'(r^j) \leq \psi(r^j)$ for all $j=1, \ldots , k$.
Minimal nontrivial valid
inequalities are associated with functions $\psi$ that are
nonnegative positively-homogeneous piecewise-linear and convex.
Furthermore, the set
\begin{equation} \label{Bpsi}
B_{\psi} := \{ x \in \mathbb{R}^2: \; \psi(x-f) \leq 1 \}
\end{equation}
is a maximal lattice-free convex set with $f$ in its interior \cite{bc}.
By {\em lattice-free convex set} we mean a convex set
with no integral point in its interior. However integral points are
allowed on the boundary. These maximal lattice-free convex sets are
splits, triangles, and quadrilaterals as proved by Lov\'asz \cite{lovasz}.
In this paper, we will focus on splits and triangles only.

A \emph{split} is a set of the form $c \le a x_1 + b x_2 \le c + 1$
where $a$ and $b$ are coprime integers and $c$ is an integer.

Following Dey and Wolsey \cite{dw}, we partition the maximal
lattice-free triangles into three types (see Figure
\ref{The4cases}):

\begin{itemize}
\item {\it Type 1 triangles}: triangles with integral vertices
and exactly one integral  point in the relative interior of each
edge;
\item {\it Type 2 triangles}: triangles with at least one fractional
vertex $v$, exactly one integral point in the relative interior of
the two edges incident to $v$ and at least two integral points on
the third edge;
\item {\it Type 3 triangles}: triangles with exactly three integral
points on the boundary, one in the relative interior of each edge.
\end{itemize}
Figure \ref{The4cases} shows these three types of triangles as well
as a maximal lattice-free quadrilateral and a split.

\begin{figure}[htbp]
\centering \includegraphics[scale=0.6]{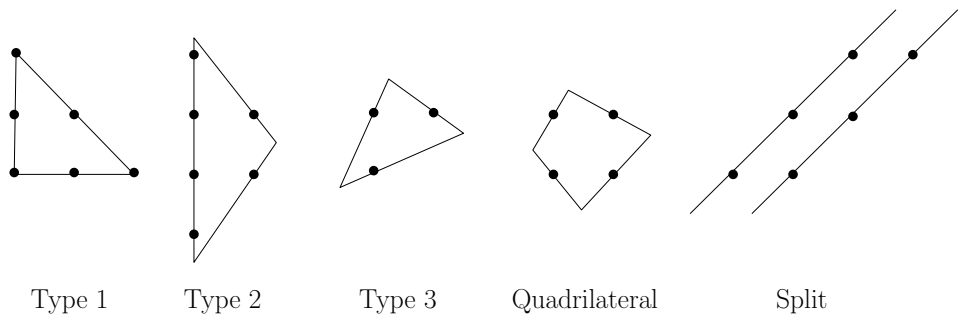}
\caption{Maximal lattice-free convex sets with nonempty interior in
$\mathbb{R}^2$} \label{The4cases}
\end{figure}

In this paper we focus on simple classes of split and Type 2 triangles.
We motivate our choice of cuts in Section \ref{sec:type2}.
In Section \ref{sec:derivation}, we give closed form formulas for our cuts.
These cuts can be seen as a particular case of intersection cuts \cite{bal}.

\subsection{Type 2 triangles versus splits}\label{sec:type2}

For model (\ref{SI}), the {\em split closure} is defined as the
intersection of all minimal valid inequalities obtained from splits
\cite{COOK}. Similarly, the {\em Type 2 triangle closure} is the
intersection of all the minimal valid inequalities generated from
Type 2 triangles. The Type 2 triangle closure approximates the
convex hull of (\ref{SI}) within a factor of 2 (see \cite{BBCM}
Theorem 7.3 for a precise statement). On the other hand, the split
closure is not always a good approximation of the convex hull of
(\ref{SI}). In \cite{BBCM}, we construct integer programs where a
single inequality derived from a Type 2 triangle closes the
integrality gap entirely whereas the split closure closes an
arbitrarily small fraction of the gap. In these examples, $f$ has
one integral and one fractional component.

These results
 suggest that it
might be interesting to generate cuts from Type 2 triangles in
addition to the classical Gomory Mixed Integer cuts \cite{nw}
when solving an MILP.
Unfortunately, it appears difficult to separate over all Type 2
triangle cuts, as it seems to be at least as hard as the NP-hard
problem of separating over the split closure \cite{CapLetch,CornLi}.
In this paper, we separate a very restricted class of Type 2 triangle cuts.
Consider the simplex tableau obtained by solving the LP relaxation of
MILP. We have
$$ x = f + \sum_{j=1}^k r^j s_j$$
\noindent where $x$ denotes the basic variables and $s$ the nonbasic.
Motivated by the examples mentioned above,
we test the effectiveness of cutting planes from two rows of the
simplex tableau
\begin{equation} \label{xixell}
\begin{array}{rrcl}
          & x_i & = & f_i + \displaystyle \sum_{j=1}^{k} r_i^j s_j\\[0.1cm]
          & x_\ell & = & f_\ell + \displaystyle \sum_{j=1}^{k} r_\ell^j s_j
\end{array}
\end{equation}
\noindent where $x_i$ and $x_\ell$ are basic integer variables,
$f_i$ is integral and $f_\ell$ is fractional, i.e. at least one of
the basic integral variables has an integral value and at least one
has a fractional value. Note that this will
typically occur when the basis is degenerate. Without loss of
generality we can make a translation so that $f_i = 0$ and $f_\ell
\in ] 0, 1[$. We try to construct two Type 2 triangles, one where
the edge that contains at least two integral points is on the line $x_i = -1$,
and a symmetrical one where this edge is on the line $x_i = 1$.
As these two cases are similar, we focus
on the first one in the following discussion. For each ray $r^j$
such that the half-line $f + \lambda r^j$ for $\lambda > 0$
intersects the plane $x_i = -1$, we compute that intersection and
its projection onto the $(x_i, x_\ell)$-plane. Among these
projections, we define $p^2 = (p^2_i, p^2_\ell)$ (resp. $p^3 =
(p^3_i, p^3_\ell)$) as the point with largest (resp. smallest)
$x_\ell$-coordinate, i.e.,
\begin{eqnarray*}
p^2 = \left(-1, f_\ell - \min_j \left\{\frac{r^j_\ell}{r^j_i} \ | \ r^j_i < 0
\right\}\right)
\hspace{1cm} \mbox{and} \hspace{1cm}
p^3 = \left(-1, f_\ell - \max_j \left\{\frac{r^j_\ell}{r^j_i} \ | \ r^j_i < 0
\right\} \right) \ .
\end{eqnarray*}
\noindent
Note that $p^2$ and $p^3$ are not defined if $r^i_j \ge 0$ for
all $j = 1, \ldots, k$.
In this case or when $p^2 = p^3$, we do not generate a cut.
Otherwise, we consider three cases:
\begin{itemize}

\item[i)] There are at least two integral points in the interior of the
segment $p^2p^3$. We construct a Type 2 triangle with
vertices $p^1$, $p^2$ and $p^3$, where $p^1$ is the intersection of the lines
passing through $p^2$ and $(0,1)$ and through $p^3$ and $(0,0)$.

\item[ii)] There is exactly one integral point in the interior of the segment
$p^2p^3$. Let $q^2 = (-1, \lceil p^2_\ell \rceil)$, and
$q^3 = (-1, \lfloor p^3_\ell \rfloor)$. If $\lceil p^2_\ell \rceil - p^2_\ell
\leq p^3_\ell - \lfloor p^3_\ell \rfloor$,
we construct a Type 1 or Type 2 triangle whose vertices are $p^1$, $q^2$ and $p^3$
where $p^1$ is the intersection of the lines passing through $q^2$ and $(0,1)$
and through $p^3$ and $(0,0)$. Otherwise, we construct a Type 2 triangle whose
vertices are $p^1, p^2$ and $q^3$, where $p^1$ is the intersection of the lines passing through $p^2$ and $(0,1)$ and through $q^3$ and $(0,0)$.

\item[iii)]
There is no integral point in the interior of the segment $p^2p^3$.
Let $q^2 = (-1, \lceil p^2_\ell \rceil)$, and
$q^3 = (-1, \lfloor p^3_\ell \rfloor)$.
We construct a split whose sides are the lines through $q^2$ and
$(0,1)$ and through $q^3$ and $(0,0)$.

\end{itemize}

Therefore, if the optimal tableau of the LP relaxation of the MILP
has $m$ basic integer variables with an integer value and $n$ basic
integer variables that are fractional, this algorithm generates up
to $2nm$ cuts from the optimal simplex tableau. We use variants of
this basic algorithm for generating and testing two-row cuts.

\section{Derivation of the cuts}
\label{sec:derivation}
\subsection{Formulas of the cut for Type 1 and Type 2 triangles}

Consider a lattice-free triangle $T$ of Type 1 or 2 with vertices
$p^1$, $p^2$ and $p^3$.
Without loss of generality, assume that $T$ is in a canonical form where
$p^2=(-1, 1 + \eta)$ and $p^3 = (-1, - \mu)$ and
$0 \le \eta \le \mu \le \eta+1$.
Furthermore we assume that the points $(0,0)$ and $(0,1)$ are
respectively in the relative interior of edge $p^3p^1$ and $p^2p^1$.
A Type 1 or Type 2 triangle can always be put into this canonical form by
applying a unimodular transformation (see Lemma 11 and Proposition 26 in
\cite{dw}).

For a Type 1 triangle, we thus have $\eta = 0$ and $\mu=1$ and the
three vertices of the triangle are $p^1=(1,0)$, $p^2=(-1,2)$ and
$p^3=(-1,0)$, see Figure \ref{fig:canonical}. For a Type 2 triangle
the points $(-1,0)$ and $(-1,1)$ are in the relative interior of the
line segment $p^2p^3$, see Figure \ref{fig:canonical}.

\begin{figure}
\begin{minipage}{0.5\textwidth}
\begin{center}
\begin{pspicture}(-2,-1)(2,3)
\pnode(1,0){p1}
\pnode(-1,2){p2}
\pnode(-1,0){p3}
\pnode(-1,0){a}
\pnode(-1,1){b}
\pnode(0,0){O}
\multido{\n=-2+1}{4}{
\multirput(\n,-1)(0,1){4}{\psdot}
}


\ncline[offset=0pt]{p1}{p2} \ncline[offset=0pt]{p2}{p3}
\ncline[offset=0pt]{p3}{p1} \nput{0}{p1}{$p^1$} \nput{90}{p2}{$p^2$}
\nput{180}{p3}{$p^3$} \nput{-70}{O}{$(0,0)$}

\end{pspicture}
\end{center}
\end{minipage}
\hspace{0.mm}
\begin{minipage}{0.5\textwidth}
\begin{center}
\begin{pspicture}(-2,-3)(2,4)
\multido{\n=-2+1}{4}{ \multirput(\n,-3)(0,1){8}{\psdot(0,0)} }

\pnode(0.2,0.5){p1} \nput{0}{p1}{$p^1$} \pnode(-1,3.4){p2}
\nput{45}{p2}{$p^2$} \pnode(-1,-2.6){p3} \nput{-10}{p3}{$p^3$}
\pnode(-1,0){a} \pnode(-1,1){b} \pnode(0,0){O}
\nput{-40}{O}{$(0,0)$}
\ncline{p1}{p2} \ncline{p2}{p3} \ncline{p3}{p1}

\ncline[offset=15pt]{|-|}{b}{p2} \lput*{0}{$\eta$}
\ncline[offset=15pt]{|-|}{p3}{a} \lput*{0}{$\mu$}

\end{pspicture}
\end{center}
\end{minipage}
\caption{\label{fig:canonical}Type 1 triangle in canonical form
($\mu = 0$, $\eta = 1$) and Type 2 triangle in canonical form ($0 <
\eta \le \mu \le \eta + 1$).}
\end{figure}
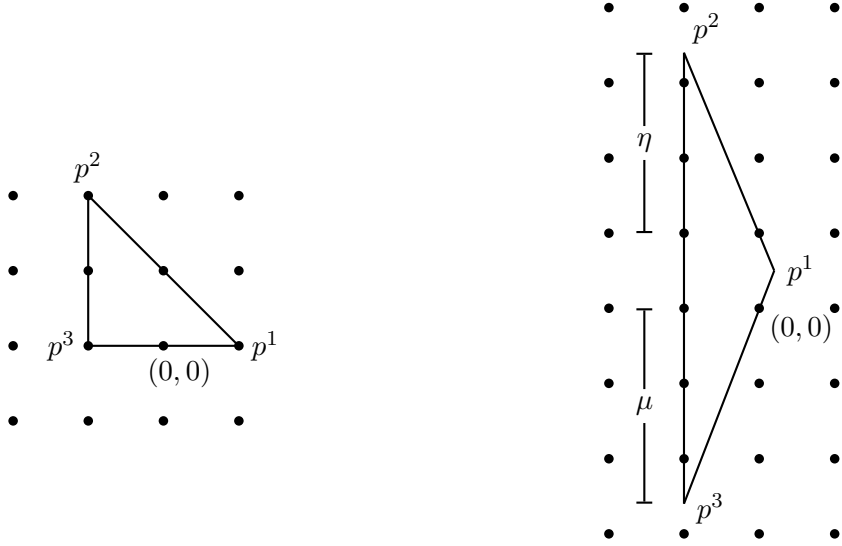

As noted in \cite{alww}, the cut obtained from the lattice-free set
$T$ is the intersection cut \cite{bal} derived from the disjunction
\begin{equation}
\label{dsj_1}
\left( - x_i \geq 1 \right)
\hspace{0.5cm} \vee \hspace{0.5cm}
\left( \eta x_i + x_\ell \ge 1 \right)
\hspace{0.5cm} \vee \hspace{0.5cm}
\left( \mu x_i - x_\ell \ge 0 \right)
\end{equation}
Disjunction (\ref{dsj_1}) simply states that all feasible points are
not in the interior of triangle $T$. Using the formula of the
intersection cut (see Appendix \ref{inter_cut}), the intersection
cut obtained from the disjunction (\ref{dsj_1}) is
\begin{equation}
\label{eq:inter_1}
\sum\limits_{j = 1}^k \max \left\{
\frac{- r^j_{i}}{1 + f_i}, \;
\frac{\eta r^j_{i} + r^j_{\ell}}{1 - \eta f_i - f_\ell}, \;
\frac{\mu r^j_{i} - r^j_{\ell}}{- \mu f_i + f_\ell}
\right\} s_j \ge 1 \ .
\end{equation}

Since $T$ is a maximal lattice-free triangle,  \eqref{eq:inter_1} is a minimal inequality for \eqref{SI}. In the next two sections, we show how the cut (\ref{eq:inter_1}) can
be strengthened if some of the non-basic variables are integer
constrained (this operation is called an integer lifting of the cut)
or if $x_i$ is nonnegative.

\subsection{Integer Lifting}
\label{sec:integer_lifting}
In this section, we consider a variation of the two-row problem \eqref{xixell}
where some of the variables $s_j$ are constrained to be integer
\begin{equation}
\label{def:2row_int}
\begin{array}{rrcl}
          & x_i & = & f_i + \displaystyle \sum_{j=1}^{k} r_i^j s_j\\[0.1cm]
          & x_\ell & = & f_\ell + \displaystyle \sum_{j=1}^{k} r_\ell^j s_j \\[0.1cm]
          &x_i, x_\ell & \in &\mathbb Z,\\[0.1cm]
          &s_j &\in & \mathbb Z \;\;\;\; \mbox{for \ all \ } j \in I \subseteq \{1,\ldots,k\},
\\[0.1cm]
          &s_j &\geq &0 \ .
\end{array}
\end{equation}

Clearly inequality \eqref{eq:inter_1} is valid for (\ref{def:2row_int}).
Dey and Wolsey \cite{dw} show how to strengthen the coefficients of
the non-basic integer variables in \eqref{eq:inter_1} to obtain a minimal inequality for
(\ref{def:2row_int}). This strengthening is called {\em integer lifting}.
If the lattice-free set used to generate the cut is a triangle of Type 1 or
2, Dey and Wolsey \cite{dw} show that there is a unique integer lifting
that produces a minimal inequality for (\ref{def:2row_int}) and it is obtained through the so-called
trivial fill-in function. In this section, we derive a closed-form
formula for this trivial lifting. We use the terminology of
intersection cuts \cite{bal}.

To derive the formula of the cut strengthened by using the trivial
fill-in function (see Appendix \ref{trivial_fill_in}), we define for any
$j \in \{1, \ldots, k\}$ and any integers $m_1$ and $m_2$:
\begin{align*}
\Phi^j_1(m_1) &= \frac{m_1 - r^j_{i}}{1 + f_i},\\
\Phi^j_2 (m_1, m_2) &= \frac{\eta\left( r^j_{i} - m_1 \right) + r^j_{\ell} - m_2}{1 - \eta f_i - f_\ell},\\
\Phi^j_3 (m_1, m_2) &= \frac{\mu \left(r^j_{i} - m_1\right) - r^j_{\ell} + m_2}{- \mu f_i + f_\ell},
\end{align*}
and
$$
\Phi^j(m_1, m_2) = \max \left\{ \Phi^j_1(m_1), \Phi^j_2(m_1, m_2), \Phi^j_3(m_1, m_2) \right\}.
$$

The strengthened cut is then given by:
\begin{displaymath}
\sum\limits_{j \in I} \min_{(m_1, m_2) \in \mathbb Z^2} \Phi^j(m_1,m_2) s_j + \sum\limits_{j \in \bar I} \Phi^j(0,0) s_j \geq 1 \ ,
\end{displaymath}
\noindent
where $ \bar I = \{1,\ldots,k\} \setminus I$.
The following proposition gives the closed form formula for
$\min_{(m_1, m_2) \in \mathbb Z^2} \Phi^j(m_1,m_2)$.
\begin{prop}\label{prop:int-lift-closed-form}
Let $\overline m^j_2 = r^j_{\ell} + \frac{\left( r^j_{i} - \lfloor
r^j_{i} \rfloor \right) \left( - \mu + \left(\mu + \eta \right)
f_\ell \right)}{1 - \left(\mu + \eta \right) f_i}$, $\overline m^j_1
= r^j_{i} - \frac{(1 + f_i) \left(  r^j_{\ell} - \lfloor r^j_{\ell}
\rfloor \right)}{1 + \eta - f_\ell}$ and $\widehat m^j_1 = r^j_{i} -
\frac{(1 + f_i) (\lceil r^j_{\ell} \rceil - r^j_{\ell})}{\mu +
f_\ell }$.

\begin{enumerate}
\item[\textit{(i)}]
$$
\min\limits_{\substack{m_1, m_2 \in \mathbb Z\\ m_1 \leq r^j_{i}}}
\Phi^j(m_1,m_2)= \min \left\{ \Phi^j_2 \left( \lfloor r^j_{i} \rfloor, \lfloor \overline m^j_2 \rfloor \right), \Phi^j_3 \left(\lfloor r^j_{i} \rfloor, \lceil \overline m^j_2 \rceil \right) \right\};
$$

\item[\textit{(ii)}]
 $$\min\limits_{\substack{m_1, m_2 \in \mathbb Z\\ m_1 > r_{1j}, m_2 \leq r^j_{\ell}}}
\Phi^j(m_1,m_2)=
\begin{cases} \min \left\{ \Phi^j_1\left( \lceil \overline m^j_1 \rceil \right),
\Phi^j_2 \left( \lfloor \overline m^j_1 \rfloor , \lfloor r^j_{\ell} \rfloor \right)
 \right\} & \text{if } \overline m^j_1 > r^j_{i},\\
 \Phi^j_2(\lceil r^j_{i} \rceil, \lfloor r^j_{\ell} \rfloor) & \text{otherwise.}
 \end{cases};
 $$
\item[\textit{(iii)}] $$\min\limits_{\substack{m_1, m_2 \in \mathbb Z\\
m_1 > r^j_{i}, m_2 \geq r^j_{\ell}}}
\Phi(m^j_1,m^j_2)= \begin{cases} \min \left\{ \Phi^j_1\left( \lceil \widehat m^j_1 \rceil \right),
\Phi^j_3 \left( \lfloor \widehat m^j_1 \rfloor , \lceil r^j_{\ell} \rceil \right)
 \right\} & \text{if } \widehat m^j_1 > r^j_{i},\\
 \Phi^j_3 \left(\lceil r^j_{i} \rceil, \lceil r^j_{\ell} \rceil \right) & \text{otherwise.}
 \end{cases}.
 $$
\end{enumerate}

\end{prop}
\begin{proof}


(i) First we show that if $m_1 \le r^j_{i}$, then
$\Phi^j(m_1, m_2) = \max \left\{\Phi^j_2(m_1, m_2), \Phi^j_3(m_1, m_2) \right\}.$ Note that if $m_1 \le r^j_{i}$, $\Phi^j_1(m_1) = \frac{m_1 - r^j_{i}}{1 + f_1} \le 0$. On the other hand, since $\eta$ and $\mu$ are nonnegative and either $m_2 - r^j_{\ell}$ or $r^j_{\ell} - m_2$ is nonnegative $\Phi^j(m_1, m_2) \ge 0$.\\
Since $\eta$ and $\mu$ are positive, if $m_1 \le \lfloor r^j_{i} \rfloor$:
\begin{align*}
\frac{\eta\left( r^j_{i} - m_1 \right) + r^j_{\ell} - m_2}{1 - \eta f_1 - f_2}
\ge \frac{\eta\left(r^j_{i} - \lfloor r^j_{i} \rfloor \right) + r^j_{\ell} - m_2}{1 - f_2 - \eta f_1}\\[0.2cm]
\frac{\mu \left(r^j_{i} - m_1 \right) - r^j_{\ell} + m_2}{f_2 - \mu f_1} \ge \frac{\mu \left(r^j_{i} - \lfloor r^j_{i} \rfloor \right) - r^j_{\ell} + m_2}{f_2 - \mu f_1}\\
\end{align*}
and therefore $\min \{\Phi^j(m_1,m_2) \ | \ m_1, m_2 \in \mathbb Z,
m_1 \leq r^j_{i}\}
= \min \{\Phi^j (\lfloor r^j_{i} \rfloor , m_2) \ | \ m_2 \in \mathbb Z\}$.
This last minimum is attained either in
$\Phi^j (\lfloor r^j_{i} \rfloor , \lceil \overline m_2 \rceil )$ or
$\Phi^j (\lfloor r^j_{i} \rfloor , \lfloor \overline m_2 \rfloor )$.
Therefore
\begin{eqnarray*}
\min \{\Phi^j(m_1,m_2) \ | \ m_1, m_2 \in \mathbb Z,
m_1 \leq r^j_{i} \}
= \min \left\{ \Phi^j_2(\lfloor r^j_{i} \rfloor, \lfloor \overline m^j_2 \rfloor ), \Phi^j_3(\lfloor r^j_{i} \rfloor, \lceil \overline m^j_2 \rceil) \right\} \ .
\end{eqnarray*}
(ii) We now suppose that $m_1 > r^j_{i}$ and $m_2 \le r^j_{\ell}$. Then
$\Phi^j_3(m_1, m_2) \le 0$ and therefore
$\Phi^j(m_1,m_2)= \max \left\{
\Phi^j_1 \left(m_1 \right), \Phi^j_2 \left(m_1,m_2 \right) \right\}$.
Furthermore, we have:
\begin{eqnarray*}
 \min \{\Phi(m_1,m_2) \ | \ m_1, m_2 \in \mathbb Z, m_1 \leq r^j_{i},
m_2 \leq r^j_{\ell}\}
= \min \{\Phi (m_1, \lfloor r^j_{\ell} \rfloor) \ | \
m_1 \in \mathbb Z, m_1 \leq r^j_{i}\} \ .
\end{eqnarray*}

If $\overline m_1 < r^j_{i}$, the minimum is attained either in $\Phi^j (\lfloor \overline m_1 \rfloor, \lfloor r^j_l \rfloor)$ or $\Phi (\lceil \overline m_1 \rceil, \lfloor r^j_l \rfloor)$, otherwise the minimum is attained in $\Phi^j (\lceil r^j_i \rceil, \lfloor r^j_l \rfloor)$. The formula follows.

(iii) Similar to the proof of (ii).
\end{proof}

\subsection{Using nonnegativity of basic variables}
\label{SEC:no_side}
In model (\ref{SI}), the two basic variables are assumed to be free, but
it may happen that one or both are constrained to be nonnegative
in the original model. This model was studied by Fukasawa and G\"unl\"uk
\cite{FG}, Dey and Wolsey \cite{dw}, and Conforti, Cornu\'ejols and
Zambelli \cite{ccz09}.

In this section, we consider the case where $x_i$ is constrained to be nonnegative,
and provide a closed-form formula to strengthen the cuts (\ref{eq:inter_1}).

Intuitively, for a triangle as described in Section \ref{sec:2row},
if we have $x_i \ge 0$
then we can remove the vertical side of the triangle and expand it
into an unbounded set containing no feasible solution of (\ref{SI}).
We call this set a {\em wedge}.

We consider a Type 1 or 2 triangle in canonical form together with
the feasible set, including the nonnegativity on $x_i$:
\begin{equation}
\label{mod:no_side}
\begin{array}{rrcl}
          & x_i & = & f_i + \displaystyle \sum_{j=1}^{k} r_i^j s_j\\[0.1cm]
          & x_\ell & = & f_\ell + \displaystyle \sum_{j=1}^{k} r_\ell^j s_j \\[0.1cm]
          &x_i, x_\ell & \in &\mathbb Z,\\[0.1cm]
                    & x_i &\geq & 0 \\[0.1cm]
          &s_j &\in & \mathbb Z \;\;\;\; \mbox{for \ all \ } j \in I \subseteq \{1,\ldots,k\},
\\[0.1cm]
          &s_j &\geq &0 \ .
\end{array}
\end{equation}

The following disjunction is valid
\begin{equation}
\label{disj_2}
\left( \eta x_i + x_\ell \ge 1 \right)
\hspace{0.5cm} \vee \hspace{0.5cm}
\left( \mu x_i - x_\ell \ge 0 \right) \ .
\end{equation}
The intersection cut obtained from (\ref{disj_2}) is given by:
\begin{equation}
\label{inter_i}
\sum\limits_{j \in J} \max \left\{
\frac{\eta r^j_{i} + r^j_{\ell}}{1 - f_\ell - \eta f_i},
\frac{\mu r^j_{i} - r^j_{\ell}}{f_\ell - \mu f_i}
\right\} s_j \ge 1.
\end{equation}
Note that the intersection cut (\ref{inter_i}) dominates
(\ref{eq:inter_1}). In particular some of the coefficients of the
latter cut may be negative. Balas \cite{balpers} observed the
following.

\begin{prop}
Inequality \eqref{inter_i} is valid for the split closure.
\end{prop}

\begin{proof}
Let $P$ be the polyhedron obtained by dropping the integrality constraints in
(\ref{mod:no_side}). Let $P_1 = P \cap \{x_\ell \geq 1\}$ and $P_2 = P \cap \{x_\ell \leq 0\}$. Inequality (\ref{inter_i}) is valid for $\left( P \cap  \{ \eta x_i + x_\ell \ge 1 \} \right)
\cup
\left( P \cap  \{ \mu x_i - x_\ell \ge 0  \} \right)$. Since $P_1 \subseteq P \cap  \{ \eta x_i + x_\ell \ge 1 \}$ and $P_1 \subseteq P \cap  \{ \mu x_i - x_\ell \ge 0  \} $, inequality (\ref{inter_i}) is valid for $P_1 \cup P_2$, which proves that it is a split
inequality.
\end{proof}

An optimal integer lifting for this cut can be
computed in a similar way to the trivial lifting computed in Section
\ref{sec:integer_lifting}. The following proposition describes how
the lifting can be adapted for disjunction (\ref{disj_2}). The
difference with the lifting performed using Proposition
\ref{prop:trivial} is that here the disjunction can be shifted by
any integer in the direction of $x_\ell$ but only by a positive
integer in the direction of $x_i$. See \cite{ccz09} for a proof that
this lifting is the unique minimal lifting.
\begin{prop}
\label{prop:str_dsj}
For $j\in I$, let $m^j \in \mathbb{Z}^2$ with $m^j_1 \ge 0$. The disjunction
\begin{equation}
\label{disj_3}
\left( \eta \left( x_i + \sum_{j \in I} m^1_j s_j\right) + x_\ell  - \sum_{j \in I} m^j_2 s_j \ge 1 \right)
\hspace{0.5cm} \vee \hspace{0.5cm}
\left( \mu \left(x_i + \sum_{j \in I} m^1_j s_j \right) -
x_\ell + \sum_{j \in I} m^j_2 s_j \ge 0 \right)
\end{equation}
is satisfied by all points $((x_i, x_\ell), s) \in \mathbb Z^2 \times
\mathbb R^n$ with $x_i \ge 0$ and $s_j \in \mathbb Z_+$ $\forall j \in I$.
\end{prop}
\begin{proof}
Suppose that $((x_i, x_\ell), s) \in \mathbb Z^2 \times \mathbb R^n$
with $x_i \ge 0$ and $s_j \in \mathbb Z$ $\forall j \in I$
does not satisfy (\ref{disj_3}). Let $x'_i= x_i + \sum_{j \in I} m^j_1 s_j$
and $x'_\ell = x_\ell - \sum_{j \in I} m^j_2 s_j$. Then
$x'_i \in \mathbb Z_+$ and $x'_\ell \in \mathbb Z$ but
$(x'_i, x'_\ell)$ does not satisfy (\ref{disj_2}).
\end{proof}
Using this proposition, we can derive a formula for a strengthened cut
which uses the integrality of the variables $s_j$ for $j \in I$.
For $m_1 \in \mathbb Z_+$ and $m_2 \in \mathbb Z$, we define
\begin{displaymath}
\Psi^j(m_1, m_2) = \max \left\{
\frac{\eta \left( r^j_{i} + m_1\right) + r^j_{\ell} - m_2}{1 - f_\ell - \eta f_i},
\frac{\mu \left( r^j_{i} + m_1 \right) + m_2 - r^j_{\ell}}{f_\ell - \mu f_i}
\right\}.
\end{displaymath}

The following cut is valid by application of Proposition
\ref{prop:str_dsj}:
\begin{equation}
\sum\limits_{j \in I} \min\limits_{(m_1,m_2) \in \mathbb Z_+ \times \mathbb Z} \Psi^j(m_1, m_2) s_j
+
\sum\limits_{j \in J \setminus I} \max \left\{
\frac{\eta r^j_{i} + r^j_{\ell}}{1 - f_\ell - \eta f_i},
\frac{\mu r^j_{i} - r^j_{\ell}}{f_\ell - \mu f_i}
\right\} s_j \ge 1
\end{equation}

The next proposition gives a closed form formula for computing the coefficients of this strengthened cut.
\begin{prop}\label{closed-form}
$$
\min\limits_{(m_1,m_2) \in \mathbb Z_+ \times \mathbb Z} \Psi^j(m_1, m_2) = \min \left\{
\frac{\eta r^j_{i} + r^j_{\ell} - \lfloor \overline m^j_2 \rfloor }{1 - f_\ell - \eta f_i},
\frac{\mu r^j_{i} + \lceil \overline m^j_2 \rceil - r^j_{\ell}}{f_\ell - \mu f_i}
\right\}
$$
with $\overline m^j_2 = r^j_{\ell} + \frac{r^j_{i}\left( - \mu + \left(\mu + \eta \right) f_\ell \right)}{1 - \left(\mu + \eta \right) f_i}$
\end{prop}
\begin{proof}
Since $\eta$ and $\mu$ are positive and $m_1 \ge 0$
\begin{align*}
\frac{\eta\left( r^j_{i} - m_1 \right) + r^j_{\ell} - m_2}{1 - \eta f_i - f_\ell}
\ge \frac{\eta r^j_{i} + r^j_{\ell} - m_2}{1 - f_\ell - \eta f_i}\\
\frac{\mu \left(r^j_{i} - m_1 \right) - r^j_{\ell} + m_2}{f_\ell - \mu f_i} \ge \frac{\mu r^j_{i} - r^j_{\ell} + m_2}{f_\ell - \mu f_i}\\
\end{align*}
and therefore
$$
\min\limits_{(m_1,m_2) \in \mathbb Z_+ \times \mathbb Z} \Psi^j(m_1, m_2) = \min\limits_{m_2 \in \mathbb Z} \Psi^j(0, m_2).
$$
The minimum is attained either in
$\Psi^j(0, \lfloor \overline m^j_2 \rfloor)$ or
$\Psi^j(0, \lceil \overline m^j_2 \rceil)$.
\end{proof}

Note that the cuts generated from the liftings of
Propositions~\ref{prop:int-lift-closed-form} and \ref{closed-form}
are incomparable, i.e. neither is guaranteed to dominate the other.
They will be compared in Section \ref{sec:compu}.

\section{Computational tests}
\label{sec:compu}

We now present computational results obtained by applying the
separation procedure devised in Section \ref{sec:type2}. The cut
generator is implemented in C++ using the COIN-OR framework
\cite{COINOR} with Cbc-2.2.2 and using Clp-1.8.2 for solving linear
programs. The machine used for all the experiments in this section
is a 64 bit {\tt Monarch Empro 4-Way Tower Server} with four {\tt
AMD Opteron 852 2.6GHz} processors, each with eight {\tt DDR-400}
SDRAM of 2 {\tt GB} and running {\tt Linux Fedora 11}. The compiler
is {\tt gcc version 4.4.0 20090506 (Red Hat 4.4.0-4)}. Results are
obtained using only one processor. Test instances for our
experiments are the sixty-eight {\tt MIPLIB3\_C\_V2}
instances~\cite{FMweb}. These instances are slight modifications of
the standard {\tt MIPLIB3} \cite{MIPLIB3} instances for which the validity
of a provided
feasible solution can be checked in finite precision arithmetic.

The main motivation behind these tests is to determine if triangle
cuts are substantially different from traditional Gomory Mixed
Integer (GMI) cuts, and if using them in conjunction with GMI cuts
might be useful. In other words, we are not suggesting to replace
GMI cuts by triangle cuts. Rather, we want to investigate if
triangle cuts can improve performance when used in addition to GMI
cuts.

\old{
\subsection{Depth comparison}
\label{SEC:depth}

As mentioned in Section~\ref{sec:type2}, if the optimal simplex
tableau has $m$ basic integer variables at integer values then for
each fractional basic variable the algorithm of
Section~\ref{sec:type2} generates $2m$ different cuts. Moreover, for
some of these triangles, we could obtain another cut by removing the
vertical side as explained in Section~\ref{SEC:no_side}. We compare
the deepest of all these possible triangle cuts corresponding to a
fractional row with the GMI cut associated with that row.

In the remaining 52 instances, the average depth of the two-row cuts
is more than the average depth of the GMI cuts (averaged over all
the instances). Table \ref{TAB:depthrm} gives more information. For
a particular instance, we measure the depth of the deepest triangle
cut as a percentage of the depth of the GMI cut from the same
fractional row. We then average this percentage over all fractional
rows to get a single percentage number associated with that
instance. We list below the number of instances associated with
ranges on this percentage. For example, in the first table the
column corresponding to 100-125 says that, for 44 instances, the
average percentage of the depth of the triangle cut is between 100
and 125 percent of the GMI cut. The particular breakpoints are
chosen so that the ranges above and below 100\% are symmetric with
respect to the ratios corresponding to these breakpoints. More
explicitly, the breakpoints for the ranges are : $0$, $\frac{1}{2}$,
$\frac{2}{3}$, $\frac{4}{5}$, $1$, $\frac{5}{4}$, $\frac{3}{2}$,
$2$.

\medskip
\begin{table}[htb]\caption{Table for depth results with cut of Proposition \ref{closed-form}}
\centering
\begin{tabular}{|c|c|c|c|c|c|c|c|}
\hline 0-50 & 50-66 & 66-80 & 80-100 & 100-125 & 125-150 & 150-200 &
200+
\\
\hline 0 & 0 & 0 & 2 & 44 & 5 & 1 & 0 \\
\hline
\end{tabular}\label{TAB:depthrm}
\end{table}

The two instances in the range 80-100 have average depth percentages
99.79\% and 99.86\%. This shows that, in general, the deepest
triangle cuts are deeper than the corresponding GMI cuts. }

\subsection{Results on gap closed by our cuts}\label{sec:G}

We compare the following four generators. As described in
Section~\ref{sec:type2}, we partition the rows of the optimal
simplex tableau corresponding to the basic integer variables into
{\it fractional rows} and {\it integer rows}.

\begin{itemize}
\item {\tt G} : A cut generator which generates one GMI cut for each fractional
row.
\item {\tt G-2rounds} : A cut generator which applies {\tt G} to generate a first round of GMI cuts, reoptimizes
the resulting LP and generates a second round of GMI cuts from the new optimal basic solution.
\item {\tt G+Allpairs} : A cut generator which generates one GMI cut for each fractional
row and all cuts derived from every pair
of fractional and integer rows as explained in
Section~\ref{sec:type2}. We also add cuts derived by removing the
vertical side of the triangle whenever possible, as outlined in
Section~\ref{SEC:no_side}.
\item {\tt G+Deepest} : {\tt G+Allpairs} often generates a large number of
cuts. To reduce this number, we only add the GMI cuts and the deepest among all
possible triangle cuts for each fractional row.
Note that {\tt G} generates one cut
per fractional row, whereas {\tt G+Deepest} generates two.
\end{itemize}

Since computers work in finite precision, cut generators sometimes
generate invalid cuts. To limit these occurrences, a variety of
tolerances and safeguards are typically used. For GMI cuts,
the most important ones are a lower bound on {\em integer
infeasibility} and an upper bound on {\em dynamism}. For a solution $\bar x$,
define the {\em integer infeasibility} of an integer variable $x_i$
as $\min\{\bar x_i - \lfloor\bar x_i\rfloor , \lceil\bar x_i\rceil -
\bar x_i\}$. We do not generate a cut from a fractional row whose basic
integer variable has an integer infeasibility smaller than 0.01. Define the
{\em dynamism} of a cut as  the ratio between the largest and the
smallest nonzero absolute values of its left hand side coefficients.
We discard generated cuts whose dynamism is larger than $10^9$.
We use the same parameters for the generation of two-row cuts with,
in addition, the condition that, for a row to be considered an
integer row, the integer infeasibility of its basic integer variable must
be at most $10^{-5}$.

A method for testing the accuracy and strength of cut generators is
given in \cite{margot}. Next we give a short description of the
method. Please refer to the original paper for additional details.

Consider an instance $I$ with a known feasible solution $x^*$. Define
a {\it cutting step} as the operation of adding cuts obtained from a
cut generator and checking if $x^*$ is still feasible (report
failure of the cut generator if $x^*$ is no longer feasible). Define
a {\it branching step} as getting an optimal solution $\bar x$ of
the current LP relaxation, picking at random an integer variable
$x_i$ with a fractional value $\bar x_i$ and imposing in the LP that
its value must be $x^*_i$.

The method, called a \emph{dive towards a feasible solution $x^*$},
amounts to repeatedly performing a cutting step followed by a
branching
step until an integer feasible solution of the LP relaxation is
obtained. In our experiments $x^*$ is chosen to be a known feasible solution whose objective value is likely to be optimum or close to optimum. A number of dives (we use twenty dives in our experiments)
are performed for each instance $I$. In our experiments, we set a time limit of 3 hours for each dive.

The motivation for using cuts in a branch-and-cut algorithm is to
improve the bound obtained from the linear relaxation and
to reduce the size of the enumeration tree.
With these two goals in mind, we track
the following during each dive:

\begin{itemize}
\item Fraction of the integrality gap closed after some given
numbers of branching steps;
\item Number of branching steps required to close some given
fractions of the integrality gap.
\end{itemize}

The random branching step allows for meaningful statistical analysis
of these two performance measures. We used the nonparametric {\it
Quade} test as described in \cite{Conover}. All tests are done with
a confidence level of 95\%. The statistical package {\tt R} \cite{R} version
2.10.0 (2009-10-26) is used for the statistical analysis of the
results.

\old{Our goal was to determine if smaller number of branching steps
are needed when we use our triangle cuts along with the GMI cuts,
when compared to the stand alone GMI cuts, to close the same
fraction of the integrality gap. The results are summarized in
Tables~\ref{TAB:diving}. }

In Table~\ref{TAB:average}, we give the average gap closed by the
four algorithms {\tt  G, G-2Rounds, G+Allpairs, G+Deepest} after 0,
4, 8 and 12 branching steps in a dive. The following problems timed
out without any results for at least one generator: 10teams\_c,
air04\_c, air05\_c, arki\_c, dano3mip\_c, fast0507\_c, mitre\_c,
mkc\_c, seymour\_c. The table gives results for the remaining 59
instances.

\old{
\begin{table}[htb]\caption{Statistical winner based on number of branching steps to close different fractions of the integrality gap}
\centering
\begin{tabular}{|c|c|c|c|c|c|}
\hline 0.5 & 0.6 & 0.7 & 0.8 & 0.9 & 1.0
\\
\hline {\tt GT} is better & {\tt GT} is better & {\tt GT} is better & {\tt GT} is better & {\tt GT} is better & {\tt G} is better  \\
\hline
\end{tabular}\label{TAB:diving}
\end{table}
}
\begin{table}[htb]\caption{Average percentage of gap closed at different depths}
\centering
\begin{tabular}{|c|c|c|c|c|}
\hline & Depth 0 (root node) & Depth 4 & Depth 8 & Depth 12
\\
\hline {\tt G} &          28.33 & 57.41 & 68.15 & 75.23 \\
\hline {\tt G+Allpairs} & 29.11 & 58.13 & 68.43 & 74.81 \\
\hline {\tt G+Deepest}  & 28.80 & 58.09 & 68.59 & 75.40 \\
\hline {\tt G-2Rounds}  & 36.66 & 59.41 & 68.75 & 75.47 \\
\hline
\end{tabular}\label{TAB:average}
\end{table}

We make the following observations from our results.
\begin{itemize}
\item Comparing {\tt G} and {\tt G+Allpairs}, we note that triangle cuts
only improve marginally over the GMI cuts at depths 0, 4 and 8. The
improvement in average gap closed at these three depths is less than
1\%. At depth 12, the average gap closed is actually more for {\tt
G}, but this difference is not statistically significant according
to the Quade test.

Even restricting to instances that are degenerate at the
root, the improvement is only 1.2\% on average at the root which is
disappointingly small given that we are generating at least four
times as many cuts.

\item Comparing {\tt G+Allpairs} and {\tt G+Deepest}, we note that most
of the above improvement at depths 0, 4 and 8 can be achieved using
only the deepest triangles. In fact, surprisingly, at depths 8 and
12 {\tt G+Deepest} closes more average gap than {\tt G+Allpairs}, and at
depth 12 the difference is statistically significant. This may
be due to the fact that adding too many cuts deteriorates the
numerical properties of the basis, and this has an adverse effect on the quality
of the cuts at greater depths.

\item Comparing {\tt G-2Rounds} and {\tt G+Deepest}, which generate roughly
the same number of cuts, we conclude that it is
clearly better to use GMI cuts at levels 0 and 4. This is confirmed
by the Quade test which shows that the difference is statistically
significant. At depths 8 and 12 the difference is not
statistically significant.
\end{itemize}

The results from Table~\ref{TAB:average} seem to indicate that our
family of two-row cuts is not competitive with GMI cuts in terms of gap closed.
This conclusion is confirmed by statistical tests performed on the
second performance measure introduced above. For brevity, the detailed results of these tests are not reported here.

We performed other experiments which seem to suggest$  $ that among the
two-row cuts that we generate, the split cuts are the more useful
ones. First, we considered two variants of {\tt G+Deepest} where the
choice of the cut is based on the length of the segment $p^2p^3$ as
defined in Section~\ref{sec:type2}. In the first variant, we chose
the cut with the longest segment $p^2p^3$; whereas, in the second we
chose the shortest. The motivation for the first choice comes from
the theoretical results in~\cite{BBCM}, where it is shown that the
gap closed by the split closure is smaller when the segment $p^2p^3$
is long. However, surprisingly, the second variant closes more gap
on average at levels 0, 4, 8 and 12, although the differences are
small. This may be explained from the fact that the second variant
generates more splits than the first and, in fact, at the root it
generates almost only splits.

In another experiment, we tried to see whether our two-row cuts
could cut off the optimum vertex of the split closure, as computed
by Balas and Saxena~\cite{balsax}. For this purpose we used the
formulations in {\tt OSCLIB 1.0} \cite{OSCLIB}.
In all but 2 of these instances, we could not cut off the split optimum.

We also compared the depths of the GMI cuts with our two-row cuts.
For 15 instances, no triangle cuts are generated at the root because the optimal
tableau is not degenerate and there are no basic integer
variables with integer values.
\old{These instances are danoint\_c,
egout\_c, fixnet3\_c, fixnet4\_c, fixnet6\_c, flugpl\_c,
markshare1\_c, markshare2\_c, mas74\_c, mas76\_c, pk1\_c,
rentacar\_c, rgn\_c, stein27\_c, stein45\_c. }For the instance
nw04\_c, the average depths are too small for any meaningful
comparison. For each of the 52 remaining instances, we measure the
depth of the deepest triangle cut as a percentage of the depth of
the GMI cut from the same fractional row. We then average this
percentage over all fractional rows to get a single percentage
number associated with that instance. For 46 out of the 52
instances, this percentage is between 99 and 125. The other 6 have a
percentage between 125 and 200. This shows that, in general, the
deepest triangle cuts are deeper than the corresponding GMI cuts.

\subsection{Importance of the Experimental Setup}

In this section, we show that one can reach drastically different
conclusions by modifying the experimental setup in a seemingly
natural way. Instead of using our own GMI cut generator, we now use
the GMI cut generator {\tt CglGomory} from the COIN-OR library with
default settings. The table corresponding to Table~\ref{TAB:average}
is as follows :

\begin{table}[htb]\caption{Average percentage of gap closed at different depths}
\centering
\begin{tabular}{|c|c|c|c|c|}
\hline & Depth 0 (root node) & Depth 4 & Depth 8 & Depth 12
\\
\hline {\tt CglGomory} & 23.83 & 51.86 & 62.91 & 69.83 \\
\hline {\tt CglGomory+All pairs} & 24.72 & 53.18 & 64.37 & 71.68 \\
\hline {\tt CglGomory+Deepest} & 24.45 & 53.92 & 64.54 & 71.19 \\
\hline {\tt CglGomory-2Rounds} & 30.89 & 53.73 & 62.85 & 69.97 \\
\hline
\end{tabular}\label{TAB:average2}
\end{table}

The conclusions we reach from this table are different from those
reached from Table~\ref{TAB:average}. We conclude that {\tt
CglGomory+All pairs} is significantly better than {\tt CglGomory} at
all depths and this advantage increases with depth, with the
difference being almost 2\% at depth 12. These differences are
statistically significant according to the Quade test. Moreover,
{\tt CglGomory+Deepest} is superior to {\tt CglGomory-2Rounds} at
depths 4, 8 and 12 with respect to gap closed and this difference is
statistically significant at depths 8 and 12. When comparison is
made on the number of branching steps needed to close different
fractions of the gap, {\tt CglGomory+Deepest} is significantly
better than {\tt CglGomory-2Rounds} for closing 70\% or more of the
gap according to the Quade test.

How do we reconcile the different conclusions reached from
Tables~\ref{TAB:average} and \ref{TAB:average2} ?

Although our goal in this paper is to compare the strength of
different cut families, in practice we can only compare cut
generators. The generator {\tt G} described in the previous section
and {\tt CglGomory} are two different generators for GMI cuts. We
described above the parameter settings for accepting the cuts in
{\tt G}. For the default {\tt CglGomory} generator the bound on
integer infeasibility is $0.05$, the bound on dynamism is $10^9$,
and there are several other parameters that may affect the
acceptance or rejection of a GMI cut. Even if we set the integer
infeasibility and dynamism bounds to the same value in both
generators, there are still significant differences in these two cut
generators. This raises the question of which generator to use for
our experiments and how to set the parameters. Depending on the
parameter settings, a cut generator may reject cuts more often
(conservative strategy) or less frequently (aggressive strategy).

A typical way that researchers test new cut generators is to add
their generator to existing base generators and evaluate the
difference in performance. In such a setting, if the base generator
is conservative and the tested generator is aggressive, one may
expect good results. On the other hand, if the base generator is
made more aggressive or the tested generator is made more
conservative, these improvements might be eroded to a large extent.
The difference between Tables~\ref{TAB:average} and
\ref{TAB:average2} is a concrete example of this phenomenon. This
issue is usually not addressed in the literature. There is a
trade-off between reliability of cut generators and their
aggressiveness. Therefore, a key aspect of such experimental setups
should involve measuring the reliability of cut generators. This was
addressed by Margot in ~\cite{margot}. The general philosophy is to
put the generator under stress and to record the number of failures
due to invalid cut generation. Concretely, the reliability of a cut
generator is tested using the diving experiment described in the
previous section where 10 rounds of cuts are generated before a
branching step is performed. According to this test, {\tt G} is
comparable in reliability to {\tt CglGomory}. In fact, {\tt G} never
generated an invalid cut for the 59 instances of {\tt
MIPLIB3\_C\_V2} considered in this paper; whereas, {\tt CglGomory}
generated invalid cuts on the instance harp2\_c.

In our view, {\tt G} should be used over {\tt CglGomory} in our
experiments because it is more aggressive and equally reliable.
Another advantage of {\tt G} over {\tt CglGomory} is that not only
are the parameter settings the same as those for the two-row cut
generators, the computer code itself is the same.
This does not necessarily imply that {\tt G} should be preferred to
{\tt CglGomory} in a general branch-and-cut setting. More experiments
would be needed to assess the reliability of the two cut generators.

This discussion shows that the setup based on {\tt CglGomory} is
flawed and that the correct setup is the one used in
Section~\ref{sec:G}, if the goal is to evaluate our two-row cuts as
compared to GMI cuts. The conclusion is that our family of two-row
cuts is not competitive with GMI cuts.

\old{
\bigskip

\medskip
\begin{table}[htb]\caption{Table for depth results with cut of Proposition \ref{prop:int-lift-closed-form}}
\centering
\begin{tabular}{|c|c|c|c|c|c|c|c|}
\hline 0-50 & 50-66 & 66-80 & 80-100 & 100-125 & 125-150 & 150-200 &
200+
\\
\hline 3 & 0 & 0 & 6 & 19 & 12 & 6 & 4\\
\hline
\end{tabular}\label{TAB:depthnorm}
\end{table}

While the difference between the two variants are not very strong, a
slight edge for the variant of Proposition
\ref{prop:int-lift-closed-form}. However, both variants produce cuts
that, on average, are deeper than the GMI cuts.}

\appendix
\section{Intersection cuts from multiple rows}
\label{inter_cut}
Proposition \ref{prop:mult_row} gives the general formula for the cut derived from a $p$-term disjunction and $m$ rows of a simplex tableau.
\begin{prop}
\label{prop:mult_row}
Let $C=\{ (x,s) \in \mathbb Z^{m}\times \mathbb R^{k} : x = f + \sum_{j =1}^k r^j s_j, s \geq 0\}$ and
$$
\Pi = \left( \pi^1 x \geq \pi^1_0 \right) \vee \ldots \vee \left(\pi^p x \geq \pi^p_0 \right)
$$
be a disjunction satisfied by all points of $\mathbb Z^m$ and not satisfied by $f$.
The inequality
\begin{equation}
\label{one} \sum\limits_{j = 1}^k \max_{l = 1,\ldots,p} \left\{
\frac{\pi^l r^j}{\pi^l_0 - \pi^l f} \right\} s_j \geq 1
\end{equation}
is valid for $C$.
\end{prop}
\begin{proof}
For each $l=1,\ldots,p$, $\pi^l x \geq \pi^l_0$ can be rewritten using the definition of $C$ in terms of the $s$ variables only:
\begin{align*}
\pi^l \left( f + \sum\limits_{j = 1}^k r^j s_j \right) \geq \pi^l_0. & & l = 1\ldots,p.
\end{align*}
Reorganizing the terms of the above inequality it can be rewritten as
\begin{align*}
\sum\limits_{j=1}^{k} \pi^l r^j s_j \geq \pi^l_0 - \pi^l f . & & l = 1\ldots,p.
\end{align*}
The point $f$ does not satisfy the disjunction and therefore $\pi^l_0 - \pi^l f > 0$, dividing the inequality by $\pi^l_0 - \pi^l f$ we obtain
\begin{align*}
\sum\limits_{j=1}^{k} \frac{\pi^l r^j}{\pi^l_0 - \pi^l f} s_j \geq 1 & & l = 1\ldots,p.
\end{align*}
Each of the above inequalities is valid for one term of the disjunction. Since they all have the same right-hand-side and $s \ge 0$, the inequality obtained by taking the component-wise maximum is valid for all terms of the disjunction and therefore
$$
\sum\limits_{j = 1}^k \max_{l = 1,\ldots,p} \left\{ \frac{\pi^l r^j}{\pi^l_0 - \pi^l f} \right\} s_j \geq 1
$$
is valid for $C$.
\end{proof}
The inequality is nothing more than the intersection cut obtained from the cone $C$ and the disjunction $\Pi$.\\

\section{Integer lifting by the trivial fill-in function}
\label{trivial_fill_in}
\begin{prop}
\label{prop:trivial} Let $C=\{ (x,s) \in \mathbb Z^{m}\times \mathbb
R^{k} : x = f + \sum_{j =1}^k r^j s_j, s \geq 0\}$. Let $I$ be a
subset of $\{1,\ldots,k\}$, $C^I=C \cap\left\{ (x,s) \in \mathbb
Z^{m} \times \mathbb R^{k} :  s_j \in \mathbb Z,\, \forall j \in I
\right\}$ be a mixed-integer set and
$$
\Pi = \left( \pi^1 x \geq \pi^1_0 \right) \vee \ldots \vee \left(\pi^p x \geq \pi^p_0 \right)
$$
be a disjunction satisfied by all $x \in \mathbb Z^m$. For $j \in I$, let $m^j$ be an integral vector of dimension $m$.
The disjunction
$$
\Pi' =
\left( \pi^1 \left( x - \sum_{j \in I} m^j s_j \right) \geq \pi^1_0 \right) \vee \ldots
\vee \left( \pi^p \left( x - \sum_{j \in I} m^j s_j \right) \geq \pi^p_0 \right)
$$
is valid for $C^I$.
\end{prop}
\begin{proof}
Suppose that the disjunction $\Pi'$ is not valid. Then there exist
$(x,s) \in \mathbb Z^m \times \mathbb R^k$ with $s_j \in \mathbb Z,
\, \forall j \in I$, such that:
\begin{align*}
\pi^l (x - \sum_{j \in I} m^j s_j ) < \pi^l_0 & & l=1,\ldots,p.
\end{align*}
Let $x' = x - \sum_{j \in I} m^j s_j$. $x' \in \mathbb Z^m$ since $x \in \mathbb Z^m$, $m^j \in \mathbb Z^m$ and $s_j \in \mathbb Z$ for all $j \in I$. Then $x'$ violates disjunction $\Pi$ which contradicts the fact that $\Pi$ is a valid disjunction.
\end{proof}

Now we consider the intersection cut obtained from a disjunction of
the form $\Pi'$ stated in Proposition~\ref{prop:trivial}.

\begin{prop}
\label{prop:mult_row_lift} Let $C=\{ (x,s) \in \mathbb Z^{m}\times
\mathbb R^{k} : x = f + \sum_{j =1}^k r^j s_j, s \geq 0\}$. Let
$J=\{1,\ldots,k\}$ and $I$ be a subset of $J$, $C^I=C \cap\left\{
(x,s) \in \mathbb Z^{m} \times \mathbb R^{k} :  s_j \in \mathbb Z,\,
\forall j \in I \right\}$ be a mixed-integer set and
$$
\Pi = (\pi^1 x \geq \pi^1_0) \vee \ldots \vee (\pi^p x \geq \pi^p_0)
$$
be a disjunction satisfied by all $x \in \mathbb Z^m$ and not satisfied by the point $f$.
For all $m^j \in \mathbb Z^m$, $j=1,\ldots,m$ the inequality
\begin{equation}
\label{two}
\sum\limits_{j \in I} \max_{l = 1,\ldots,p} \left\{ \frac{\pi^l (r^j - m^j)}{\pi^l_0 - \pi^l f} \right\} s_j +
\sum\limits_{j \in J \setminus I} \max_{l = 1,\ldots,p} \left\{ \frac{\pi^l r^j}{\pi^l_0 - \pi^l f} \right\} s_j\geq 1
\end{equation}
is valid for $C^I$.
\end{prop}
\begin{proof}
Since $\Pi$ is valid, for all $m^j \in \mathbb Z^m$, $j=1,\ldots,m$ the disjunction
$$
\Pi' =
(\pi^1 (x - \sum_{j \in I} m^j s_j ) \geq \pi^1_0) \vee \ldots
\vee (\pi^k (x - \sum_{j \in I} m^j s_j ) \geq \pi^p_0)
$$
is valid.

The terms of the disjunction can be rewritten as:
\begin{align*}
\sum\limits_{j \in I} \pi^l (r^j - m^j) s_j +
\sum\limits_{j \in J \setminus I} \pi^l r^j s_j\geq \pi^l_0 - \pi^l f. & & l = 1\ldots,p.
\end{align*}
The point $f$ does not satisfy the disjunction and therefore $\pi^l_0 - \pi^l f > 0$, dividing the inequality by $\pi^l_0 - \pi^l f $ we obtain
\begin{align*}
\sum\limits_{j \in I} \frac{\pi^l (r^j - m^j)}{\pi^l_0 - \pi^l f } s_j +
\sum\limits_{j \in J \setminus I} \frac{\pi^l r^j}{\pi^l_0 - \pi^l f} s_j \geq 1 & & l = 1\ldots,p.
\end{align*}
Each of the above inequalities is valid for one term of the disjunction. By applying, the disjunctive principle we obtain (\ref{two}).
\end{proof}

\begin{thebibliography}{99}

\bibitem{alww} K. Andersen, Q. Louveaux, R. Weismantel and L. Wolsey, Cutting Planes from Two Rows of a Simplex Tableau, {\it
Proceedings of IPCO XII}, Ithaca, New York (June 2007), Lecture
Notes in Computer Science 4513, 1-15.

\bibitem{balpers} E. Balas, personal communication.

\bibitem{bal} E. Balas, Intersection Cuts - A New Type of Cutting Planes for
Integer Programming, {\it Operations Research 19} (1971), 19--39.

\bibitem{bcc} E. Balas, S. Ceria and G. Cornu{\'e}jols, A
lift-and-project cutting plane algorithm for mixed 0-1 programs,
{\it Mathematical Programming 58\/} (1993) 295--324.

\bibitem{balsax} E. Balas and A. Saxena, Optimizing over the Split
Closure, {\it Mathematical Programming A 113} (2008) 219--240.

\bibitem{BBCM} A. Basu, P. Bonami, G. P. Cornuejols and F. Margot,
On the Relative Strength of Split, Triangle and Quadrilateral Cuts,
to appear in {\it Mathematical Programming} (2009),
DOI:10.1007/s10107-009-0281-x.

\bibitem{MIPLIB3} Bixby R.E., Ceria S., McZeal C.M., Savelsbergh M.W.P, 
MIPLIB 3.0,
{\tt http://www.caam.rice.edu/$\sim$bixby/miplib/miplib.html}.

\bibitem{bc} V. Borozan and G. Cornu\'ejols, Minimal Valid
Inequalities for Integer Constraints, technical report (July 2007,
revised August 2008).

\bibitem{CapLetch} A. Caprara and A.N. Letchford,
On the separation of split cuts and related inequalities, {\em
Mathematical Programming 94} (2003), 279--294.


\bibitem{ccz09} M. Conforti, G. Cornu\'ejols and G. Zambelli,
A Geometric Perspective on Lifting, technical report, April 2009.

\bibitem{Conover} W.J. Conover,
Practical Nonparametric Statistics, 3rd edition, Wiley (1999).

\bibitem{COOK}
W. Cook, R. Kannan and A. Schrijver, Chv\'atal Closures for Mixed
Integer Programming Problems, {\em Mathematical Programming 47}
(1990), 155--174.

\bibitem{CornKaramLi} G. Cornu\'ejols, M. Karamanov and Y. Li,
Early estimates of the size of branch-and-bound trees,
INFORMS J. Computing 18 (2006), 86--96.

\bibitem{CornLi} G. Cornu\'ejols and Y. Li,
A connection between cutting plane theory and the geometry of numbers,
{\it Mathematical Programming 93} (2002), 123--127.

\bibitem{cm} G. Cornu\'ejols and F. Margot,
On the Facets of Mixed Integer Programs with Two Integer
Variables and Two Constraints,
{\it Mathematical Programming 120} (2009), 419--456.

\bibitem{cjp} H. Crowder, E.L. Johnson and M. Padberg,
Solving Large-Scale Zero-One Linear Programming Problems,
{\em Operations Research 31} (1983), 803-834.

\bibitem{dw} S.S. Dey and L.A. Wolsey,
Lifting Integer Variables in Minimal Inequalities Corresponding
to Lattice-Free Triangles, {\it IPCO 2008}, Bertinoro, Italy (June
2008), Lecture Notes in Computer Science 5035 (2008), 463--476.

\bibitem{dw09} S.S. Dey and L.A. Wolsey,
Constrained Infinite Group Relaxations of MIPs, technical report (March 2009).

\bibitem{esp} D.G. Espinoza, Computing with multirow Gomory cuts,
{\it Proceedings of IPCO 2008}, Bertinoro, Italy, (June 2008),
Lecture Notes in Computer Science 5035 (2008), 214--224.

\bibitem{FG} R. Fukasawa and O. G\"unl\"uk, Strengthening lattice-free cuts
using non-negativity, May 2009, available on OptimizationOnline.

\bibitem{go} R.E. Gomory, An Algorithm for Integer Solutions to Linear
Programs, {\it Recent Advances in Mathematical Programming}, R.L.
Graves and P. Wolfe eds., McGraw-Hill, New York (1963),  269--302.

\bibitem{gomory} R.E. Gomory, Some Polyhedra related to
combinatorial problems, {\it Linear Algebra and its Applications 2}
(1969), 451-558.

\bibitem{gj} R.E. Gomory and E.L. Johnson,
Some Continuous Functions Related to Corner Polyhedra, Part I,
{\it Mathematical Programming 3} (1972), 23--85.

\bibitem{sc} R.E. Gomory and E.L. Johnson,
T-space and Cutting planes,
{\it Mathematical Programming, Series B 96} (2003), 341--375.

\bibitem{johnson} E.L. Johnson, On the Group Problem for Mixed
Integer Programming, {\it Mathematical Programming Study 2} (1974),
137--179.

\bibitem{lovasz} L. Lov\'asz, Geometry of Numbers and Integer Programming,
{\it Mathematical Programming: Recent Developments and
Applications}, M. Iri and K. Tanabe eds., Kluwer (1989) 177--210.

\bibitem{mw} H. Marchand and L.A. Wolsey, Aggregation and mixed integer
rounding to solve MIPs, {\em Operations Research 49} (2001), 363--371.

\bibitem{margot} F. Margot, Testing cut generators for mixed integer
linear programming, {\it Mathematical Programming Computation 1}
(2009) 69--95.

\bibitem{meyer} R.R. Meyer, On the Existence of Optimal Solutions to
Integer and Mixed Integer Programming Problems, {\it Mathematical
Programming 7} (1974), 223--235.

\bibitem{nw} G.L. Nemhauser and L.A. Wolsey, A Recursive Procedure to Generate
All Cuts for 0-1 Mixed Integer Programs, {\it Mathematical Programming
46} (1990), 379--390.

\bibitem{sch} A. Schrijver, {\it Theory of Linear and Integer
Programming}, Wiley, (1986), p.114.

\bibitem{HANDPARAM} D.J. Sheskin,
{\it Parametric and Nonparametric Statistical Procedures}, 2nd Ed.,
Chapman \& Hall/CRC (2000).

\bibitem{COINOR}
COmputational INfrastructure for Operations Research (COIN-OR)
{\tt http://www.coin-or.org}.

\bibitem{FMweb} {\tt http://wpweb2.tepper.cmu.edu/fmargot/}.

\bibitem{OSCLIB} {\tt http://www.andrew.cmu.edu/user/anureets/osc/osc.htm}

\bibitem{R} {\tt R} statistical software, {\tt http://www.r-project.org/}.

\end{thebibliography}
\end{document}